\documentclass[11pt,a4paper]{amsart}
	
	\usepackage{amssymb,amsthm}
	\usepackage{mathrsfs} 
	\usepackage{colonequals}
	
	\usepackage[utf8]{inputenc}
	\usepackage[OT2,T1]{fontenc}
	\usepackage[english]{babel}
	\DeclareSymbolFont{cyrletters}{OT2}{wncyr}{m}{n}
	\DeclareMathSymbol{\Sha}{\mathalpha}{cyrletters}{"58}
	
	\usepackage{csquotes}
	
	\usepackage[protrusion=false]{microtype}
	
	\usepackage{xcolor}
	\definecolor{darkgreen}{rgb}{0,0.5,0}
	\definecolor{darkblue}{rgb}{0,0,0.8}

	\usepackage[
	colorlinks, citecolor=darkgreen,
	linkcolor=darkblue,
	pagebackref,
	unicode=true,
	pdfauthor={Timo Keller},
pdftitle={},
]{hyperref}

\renewcommand{\epsilon}{\varepsilon}
\renewcommand{\phi}{\varphi}
\renewcommand{\theta}{\vartheta}

\newcommand{\Q}{{\mathbf{Q}}}

\newcommand{\Z}{{\mathbf{Z}}}

\newcommand{\F}{{\mathbf{F}}}

\newcommand{\Bcal}{{\mathscr{B}}}
\renewcommand{\P}{{\mathbf{P}}}

\newcommand{\Jac}{\operatorname{Jac}}

\newcommand{\Ocal}{{\mathcal{O}}}

\newcommand{\Ccal}{{\mathscr{C}}}

\newcommand{\Spec}{\operatorname{Spec}}

\newcommand{\Acal}{\mathscr{A}}
\newcommand{\NS}{\operatorname{NS}}

\newcommand{\Tr}{\operatorname{Tr}}
\newcommand{\rk}{\operatorname{rk}}

\newcommand{\Pic}{\operatorname{Pic}}
\newcommand{\hhat}{{\hat{h}}}

\newcommand{\isoto}{\xrightarrow{\sim}}

\newcommand{\inj}{\hookrightarrow}
\newcommand{\surj}{\twoheadrightarrow}
\newcommand{\defeq}{\colonequals}

\theoremstyle{plain}
\newtheorem{theorem}{Theorem}[section]
\newtheorem{lemma}[theorem]{Lemma}
\newtheorem{corollary}[theorem]{Corollary}

\theoremstyle{definition}
\newtheorem{definition}[theorem]{Definition}

\theoremstyle{remark}
\newtheorem{remark}[theorem]{Remark}
\newtheorem{notation}[theorem]{Notation}

\usepackage{cleveref}
\crefname{theorem}{Theorem}{Theorems}
\crefname{lemma}{Lemma}{Lemmata}
\crefname{corollary}{Corollary}{Corollaries}
\crefname{proposition}{Proposition}{Propositions}
\crefname{definition}{Definition}{Definitions}
\crefname{conjecture}{Conjecture}{Conjectures}
\crefname{question}{Question}{Questions}
\crefname{example}{Example}{Examples}
\crefname{algorithm}{Algorithm}{Algorithms}
\crefname{remark}{Remark}{Remarks}
\crefname{assumption}{Assumption}{Assumptions}

\usepackage{tikz-cd}

\begin{document}
	\title{Specialization of Mordell-Weil ranks\\ of abelian schemes over surfaces to curves}
	
	\author[Keller, T.]{Timo Keller}
	\address{Lehrstuhl Mathematik~II (Computeralgebra)\\Universität Bayreuth\\Universitätsstraße 30\\95440 Bayreuth, Germany} 
	\curraddr{Leibniz Universität Hannover, Institut für Algebra, Zahlentheorie und Diskrete Mathematik, Welfengarten 1, 30167 Hannover, Germany}
	\email{\href{mailto:keller@math.uni-hannover.de}{keller@math.uni-hannover.de}}
	\urladdr{\url{https://www.timo-keller.de}}
	
	\thanks{The author was partially supported by the Deutsche Forschungsgemeinschaft (DFG), Projektnummer STO 299/18-1, AOBJ: 667349 while working on this article.}
	
	\date{\today}
	
	\begin{abstract}
	Using the Shioda-Tate theorem and an adaptation of Silverman's specialization theorem, we reduce the specialization of Mordell-Weil ranks for abelian varieties over fields finitely generated over infinite finitely generated fields $k$ to the the specialization theorem for N\'eron-Severi ranks recently proved by Ambrosi in positive characteristic. More precisely, we prove that after a blow-up of the base surface $S$, for all vertical curves $S_x$ of a fibration $S \to U \subseteq \P^1_k$ with $x$ from the complement of a sparse subset of $|U|$, the Mordell-Weil rank of an abelian scheme over $S$ stays the same when restricted to $S_x$.
	
	
	\emph{2020 Mathematics} subject classification: 11G10 (primary); 11G05; 11G35
	\end{abstract}
	
	\maketitle
	
	\pagestyle{plain}
	
	\section{Introduction}
	
	\subsection{Motivation and overview}
	
	Silverman's specialization theorem~\cite{Silverman1983} states that for a family $\Acal$ of abelian varieties fibered over a curve $C$ over a number field $K$, the specialization homomorphism from the generic fiber $\Acal(K(C))$ to a special fiber $\Acal(\kappa(x))$ is injective for $x \in C(\overline{K})$ outside a subset of bounded height. (The specialization theorem has been extended to finitely generated fields in characteristic $0$ by~\cite{Wazir2006}.) In particular, if one can construct such a family with generic Mordell-Weil rank $r$, one gets infinitely many specializations over number fields of rank at least $r$. If moreover $C$ has infinitely many $K$-points, one even gets infinitely many abelian varieties over $K$ of rank at least $r$. This has been used to construct elliptic curves of high rank over $\Q$, see for example~\cite{Shioda1991}.
	
	However, the rank can be strictly larger. For example, an open problem seems to be to construct infinitely many non-isomorphic elliptic curves over $\Q$ with rank~$2$. In this article, we prove that one gets infinitely many specializations with Mordell-Weil rank \emph{equal} to the generic Mordell-Weil rank. However, our situation is slightly different: We do not specialize from a family over a curve to its closed points, but from a family over a surface to curves on that surface. More precisely, our main theorem is:
	
	\begin{theorem}[\cref{Mordell-Weil specialization}] \label[theorem]{main theorem}
		Let $k$ be an infinite finitely generated field. Let $K|k$ be a finitely generated regular field extension and $Sk$ a smooth separated (not necessarily proper) geometrically connected surface over $k$ with function field $K$. Let $A$ be an abelian variety over $K$ with $\Tr_{K|k}(A) = 0$ or $\Tr_{K|k}(A)(k)$ finite, and $\Acal$ an extension of $A$ to an abelian scheme over a dense open subscheme $U$ of $S$.
		
		Then for infinitely many curves $C$ on $U$, one has a specialization isomorphism
		\[
			A(K) \otimes \Q \isoto \Acal(C) \otimes \Q
		\]
		of rationalized Mordell-Weil groups. More precisely, for all fibrations of $S$ over a curve $U$, there is a $d \geq 1$ such that one can take infinitely many vertical curves $S_x$ for $x \in |U|$ of degree $[\kappa(x) : k] \leq d$.
	\end{theorem}
	
	Our proof combines the specialization theorem for N\'eron-Severi ranks recently proved by Ambrosi~\cite[Corollary~1.7.1.3]{ambrosi2018specialization} with the Shioda-Tate formula from~\cite{Gordon}. We then adapt Silverman's specialization theorem to prove~\cref{main theorem}.
	
	\subsection{Structure of our work}
	In section~\ref{sec:specialization-of-mordell-weil-groups-the-rank-does-not-drop} we observe that Silverman's specialization theorem~\cite{Silverman1983} holds in our setting, too, because we have the usual height machine from~\cite{Conrad-trace}. Section~\ref{sec:specialization-of-mordell-weil-groups} is the core of this article; here, we prove that after blowing up the base surface $S$, there is a proper generically smooth fibration $S \to \P^1$ (\cref{fibered surface}) such that infinitely many vertical curves $S_x$ have the property that $\Acal(S) \otimes \Q \to \Acal(S_x) \otimes \Q$ is an isomorphism, first for Jacobians in~\cref{Mordell-Weil specialization rank does not grow} and then in the general case in~\cref{Mordell-Weil specialization}.
	
	\subsection{Relation to other work}
	We originally intended to use~\cref{main theorem} to verify the missing hypothesis in~\cite[Theorem~7.0.3]{Keller-BSD}. However, the reduction of the BSD conjecture for all abelian varieties over function fields of any transcendence degree over $k$ to that of $1$-dimensional function fields is already contained in~\cite[Corollary~5.4]{geisser_2019}.
	
	The conjecture of Birch and Swinnerton-Dyer (BSD) has has been extended to function field of arbitrary transcendence degree over finite and finitely generated fields of positive characteristic in~\cite{Keller-BSD} and~\cite{Qin2}, respectively.	As the conjecture is more accessible when the (absolute or relative over a finitely generated field) transcendence degree of $K$ is $1$, it is desirable to investigate the behavior of BSD for $A/K$ when $K$ is specialized to a finitely generated field of lower transcendence degree. More geometrically, this corresponds to restricting the abelian scheme over a surface model of the function field to curves on this surface.
	
	In~\cite[§\,7]{Keller-BSD} we used a Lefschetz hyperplane argument to prove that the BSD conjecture for all abelian schemes $\Acal$ over all \emph{surfaces} $S$ over finite fields implies the BSD conjecture for all abelian schemes over all bases of dimension greater than two. The crucial property used was that the rank of the restriction of the Mordell-Weil group $\Acal(X)$ to an ample smooth irreducible hypersurface section $Y$ of $X$ of dimension $> 2$ remains invariant. However, the reduction to the case where the base is a \emph{curve} could not be completed because we could not construct curves $C$ on the surface $S$ such that the ranks of $\Acal(S)$ and $\Acal(C)$ are equal. In the present article, we fill this gap by showing that, possibly after blowing up $S$, there are infinitely many such curves. Note that the truth of the BSD conjecture is invariant under blow-ups.

\begin{notation}
	The set of closed points of a scheme $X$ is denoted by $|X|$. For a point $v$ of a scheme, we denote its residue field by $\kappa(v)$.
\end{notation}

\section{The rank does not drop outside a set of bounded height}\label{sec:specialization-of-mordell-weil-groups-the-rank-does-not-drop}

For the Shioda-Tate formula~\cref{Shioda-Tate}, we need the notion of a fibered surface from~\cite[2.1]{Gordon}; as we also need fibrations into curves over bases of dimension $\leq 2$, we introduce the notion of a relative curve:

\begin{definition}\label[definition]{def:fibered surface}
	A \emph{relative curve} $\Ccal \to S$ over a field $k$ consists of the following data: an irreducible variety $S$ over $k$, a proper smooth variety $\Ccal$ over $k$ and a proper flat morphism $\Ccal \to S$ cohomologically flat in dimension $0$ with fibers of dimension~$1$ and smooth projective geometrically irreducible generic fiber. We call a relative curve \emph{smooth} if the morphism $\Ccal \to S$ is.
	
	If $\dim{S} = 1$ and $S$ is smooth, projective, and geometrically irreducible, we call such a relative curve a \emph{fibered surface}.
\end{definition}

\begin{remark}
	That $\pi: \Ccal \to S$ is cohomologically flat in dimension $0$ means that one has $\Ocal_S \isoto \pi_*\Ocal_\Ccal$ universally. If the proper flat morphism $\pi$ admits a section, one can omit the word \enquote*{universally}. See the remarks after the definition in~\cite[2.1]{Gordon}.
\end{remark}

We use the $K|k$-trace of an abelian variety $A$ over $K$, see~\cite{Conrad-trace} and~\cite[4.2]{Gordon}:

\begin{definition}
	Let $K|k$ be a primary field extensions and $A$ an abelian variety over $K$. A \emph{$K|k$-trace} of $A$, denoted $\Tr_{K|k}(A)$, is a final object in the category of pairs $(B,f)$ where $B$ is an abelian variety over $k$ and $f: B \times_k K \to A$ is a morphism of abelian varieties.
\end{definition}

\begin{remark}
	The $K|k$-trace exists by~\cite[Theorem~6.2]{Conrad-trace}. It somewhat captures the constant part of $A$, i.e., the part coming from $k$, see~\cite[Example~2.2]{Conrad-trace}.
\end{remark}

\begin{theorem}
	Let $\Ccal \to S$ be a fibered surface with generic fiber $C \to \Spec{K}$. Assume that the field extension $K|k$ is regular (primary in~\cite[4.2]{Gordon}). Then the $K|k$-trace of $A \defeq \Pic^0_{C/K}$ is an abelian variety over $k$ purely inseparably isogenous to $\Pic^0_{\Ccal/k}/\Pic^0_{S/k}$.
\end{theorem}
\begin{proof}
	See~\cite[Proposition~4.4]{Gordon}.
\end{proof}

For an abelian scheme $\Acal \to S$, we call the abelian group $\Acal(S)$ of its sections its \emph{Mordell-Weil group}. The following is a generalization of the Mordell-Weil theorem for abelian varieties over global fields.

\begin{theorem}[Mordell-Weil-N\'eron-Lang]\label[theorem]{Mordell-Weil-Neron-Lang}
	Let $K|k$ be a finitely generated regular field extension and $A$ an abelian variety over $K$. Then \[A(K)/\Tr_{K|k}(A)(k)\] is a finitely generated abelian group.
\end{theorem}
\begin{proof}
	See~\cite[Theorem~7.1]{Conrad-trace}.
\end{proof}

As in~\cite[text before Proposition~1]{Wazir2006}, the \emph{height} of an effective divisor on a smooth projective variety with respect to an embedding in projective space is its degree as a closed subvariety of projective space.
\begin{theorem}\label[theorem]{restriction to closed subscheme injective}
	Let $K|k$ be a finitely generated regular field extension with smooth projective model $S$ over $k$, $A$ an abelian variety over $K$ with $\Tr_{K|k}(A) = 0$ or $\Tr_{K|k}(A)(k)$ finite (e.g., $k$ finite), and $\Acal$ an extension of $A$ to an abelian scheme over a dense open subscheme $U$ of $S$. For all $M > 0$, all but finitely many curves $C \inj U$ of degree $\leq M$ have the property that the specialization morphism $A(K) \otimes \Q \to \Acal(C) \otimes \Q$ is injective. In particular,
	\begin{equation} \label{eq:Silverman specialization}
		\rk A(K) \leq \rk \Acal(C).
	\end{equation}
\end{theorem}
\begin{proof}
The theorem for $k$ of characteristic $0$ is~\cite[Theorem~1 and the text before Proposition~1]{Wazir2006}. We merely describe the necessary changes when $k$ is of positive characteristic: We only have to see that we have the \enquote*{height machine} for the arithmetic and geometric height in~\cite{Wazir2006} (which generalizes Silverman's specialization theorem~\cite{Silverman1983}).

The properties~\cite[Proposition~2]{Wazir2006} of the \enquote*{arithmetic height machine} can be found for Conrad's generalized global fields in~\cite[text after Theorem~9.3]{Conrad-trace} (note that~(vi) is an immediate consequence of \enquote*{quasi-equivalence} since a curve has N\'{e}ron-Severi group $\Z$). The \enquote*{canonical height machine}~\cite[Proposition~3]{Wazir2006} for abelian varieties can be found in~\cite[text after Example~9.5]{Conrad-trace}.

The required properties of the \enquote*{geometric} height are proved in~\cite[Chapter~6, Theorem~5.4]{Lang1983} with the Northcott property in~\cite[Chapter 3, Theorem~3.6]{Lang1983} or~\cite[Lemma~10.3]{Conrad-trace}.

We sketch the proof: As almost all curves on $A$ can be realized as horizontal curves with respect to finitely many fibrations of $A$ over curves (possibly after blowing up $A$)~\cite[Proposition~1]{Wazir2006}, suppose we are in such a situation: Assume $A$ is fibered as $\rho: A \to C$ over a curve $C$. Fix a line bundle $L$ on $A$ and denote by $D_\rho$ its restriction to the generic fiber $A_\rho$ of $\rho$. As a consequence of these properties, one gets formally in the same way as in~\cite[Theorem~3]{Wazir2006} with the notation from there the equation
\[
	\lim_{t \in |C|, h^\text{arith}_C(t) \to \infty}\frac{\hhat^\text{arith}_{(A_t,L_t)}(P_t)}{h^\text{arith}_C(t)} = \hhat^\text{geom}_{(A_\rho,D_\rho)}(P_\rho)
\]
for a section $P$ of $A/C$, independently of the choice of the arithmetic height on $C$. The theorem follows from this and the Northcott property: For $h^\text{arith}_C(t) \gg 0$, $\hhat^\text{arith}_{(A_t,L_t)}(P_t)$ must be $> 0$ if $\hhat^\text{geom}_{(A_\rho,D_\rho)}(P_\rho) > 0$, but points of non-zero height are non-torsion.
\end{proof}

\section{The rank does not grow outside a sparse set}\label{sec:specialization-of-mordell-weil-groups}

In this section, $S$ is a smooth, not necessarily proper, geometrically connected surface over a finitely generated field $k$ of positive transcendence degree over its prime field if it has positive characteristic, and $\Ccal \to S$ a proper smooth morphism with fibers geometrically connected curves (a smooth relative curve). We prove our main result on the specialization of Mordell-Weil ranks first for Jacobians in~\cref{Mordell-Weil specialization rank does not grow} and then for general abelian varieties in~\cref{Mordell-Weil specialization} using the inequality~\eqref{eq:Silverman specialization} from the previous section.

Using the following lemma, we can assume that there is a smooth fibration of our surface $S/k$ to a non-empty open subscheme $U$ of $\P^1_k$:

\begin{lemma}\label[lemma]{fibered surface}
	Let $k$ be an infinite field and $S$ a smooth separated geometrically connected surface over $k$ (not necessarily proper). There is a blow-up $\tilde{S} \to \overline{S}$ with $S \inj \overline{S}$ a proper compactification of $S$ such that $\tilde{S}$ admits a proper flat morphism to $\P^1_k$ with smooth projective geometrically connected generic fiber.
\end{lemma}
\begin{proof}
	Since $S$ is a smooth separated geometrically connected \emph{surface}, it has a smooth projective compactification $S \inj \overline{S}$ by Nagata compactification and resolution of singularities of surfaces, which can be achieved using successive blow-ups of smooth centers. Now the theory of Lefschetz pencils~\cite[Expos\'ee~XVII, Th\'eor\`eme~2.5.2]{SGA72} (in characteristic~$0$; over finite fields, use~\cite[Theorem~2.2]{JannsenSaito2012}) gives a blow-up of $\overline{S}$ together with a proper morphism to $\P^1_k$ with smooth generic fiber.
\end{proof}

(A stronger version of~\cref{fibered surface} appears in~\cite[Proposition~1]{Wazir2006}, which holds for all infinite fields. Note that it is \emph{not} true that one can realize almost all smooth irreducible curves on $S$ as fibers of a fibration: For example, suppose $S$ smooth and projective with a smooth projective fibration $\pi: S \to U$ such that $C = S \times_U \{x\}$. Then one has $K_S|_{C} = K_C$ for the canonical divisor classes by the adjunction formula.)

Now restrict to a non-empty open subscheme $U$ of $\P^1_k$ over which $\widetilde{S} \to \P^1_k$ is smooth. The proper smooth relative curve $\Ccal \to S$ can be extended to a proper smooth relative curve $\widetilde{\Ccal} = \Ccal \times_S \widetilde{S} \to \widetilde{S}$ by functoriality of blow-ups~\cite[\href{https://stacks.math.columbia.edu/tag/085S}{Lemma 085S}]{stacks-project}; this does not change the generic fiber. Remove from $U$ the closed subscheme of points $x$ such that $\Ccal \times_U \{x\} =: \Ccal_x \to S_x \defeq S \times_U \{x\}$ is not smooth. This subset is not equal to $U$ because $\Ccal \to S \to U$ is generically smooth.

In the following, assume $S/k$ is a smooth geometrically connected surface admitting a proper flat morphism to a non-empty open subscheme $U$ of $\P^1_k$ with smooth and geometrically connected generic fiber. Consider the following situation, where all vertical arrows are smooth proper morphisms of relative dimension~$1$ and the squares are fiber product squares:
\begin{center}
	\begin{equation} \label{eq:situation}
	\begin{tikzcd}
		\Ccal_{k(U)} \ar[r,hookrightarrow]\ar[d] 	& \Ccal \ar[d] \ar[r,hookleftarrow]& \Ccal|_{S_x} \ar[d]\\
		S_{k(U)} \ar[r,hookrightarrow]\ar[d] 		& S \ar[d]	   \ar[r,hookleftarrow]& S_x \ar[d]\\
		\Spec{k(U)} \ar[r,hookrightarrow] 			& U \ar[d]	   \ar[r,hookleftarrow]& \{x\} \\
		& \Spec{k}
	\end{tikzcd}
	\end{equation}
\end{center}

As above, $U \subseteq \P^1_k$ is a non-empty open subscheme such that $S|_U \to U$ is smooth; such an $U$ exists because $S \to \P^1_k$ is generically smooth. In the following, we denote the restriction $S|_U$ again by $S$ and the function field of $U$ by $k(U)$. The right hand side of the diagram is constructed below; $S_x$ is going to be the fiber of $S \to U$ over a closed point $x \in |U|$, and hence a smooth projective vertical curve in $S$ over $U$.

We denote the rank of the N\'eron-Severi group of a variety $X$ by $\rho(X)$; it is finite by~\cite[Exp.~XIII, §\,5]{SGA6}.

\begin{lemma}[Specialization of N\'eron-Severi rank]\label[lemma]{NS specialization}
	Assume $k$ is a finitely generated field of transcendence degree $\geq 1$ over $\F_p$ or of characteristic $0$ (i.e., an infinite finitely generated field).
	
	Then there exists a $d \geq 1$ such that for infinitely many closed points $x$ of $U$ with $[\kappa(x):k] \leq d$ (the complement of a \emph{sparse} subset of $U$), the N\'eron-Severi rank of the special fiber $\Ccal_x \defeq \Ccal \times_U \Spec{\kappa(x)}$ equals the N\'eron-Severi rank of the generic fiber $\Ccal_{k(U)} \defeq \Ccal \times_U \Spec{k(U)}$ of the smooth proper relative surface $\Ccal \to U$: $\rho(\Ccal_x) = \rho(\Ccal_{k(U)})$.
\end{lemma}
\begin{proof}
	This follows from~\cite[Corollary~1.7.1.3\,(1)]{ambrosi2018specialization} (for $k$ finitely generated of positive transcendence degree over $\F_p$) and~\cite[Corollary~5.4]{Cadoret2013} (for $k$ of characteristic $0$) applied to the smooth proper morphism $\Ccal \to U$ with $U$ a smooth and geometrically connected $k$-curve (a non-empty open subscheme of $\P^1_k$).
\end{proof}

We now apply the Shioda-Tate formula to the surfaces $\Ccal_x \to \{x\}$ and $\Ccal_{k(U)} \to \Spec{k(U)}$ fibered over the curves $S_x$ and $S_{k(U)}$, respectively (note that these are indeed surfaces as the absolute dimension of $\Ccal$ is~$3$; see also~\eqref{eq:situation} for the setting). This allows us to translate this equality of the N\'eron-Severi ranks of the generic and special fibers into an (in)equality of (between) the Mordell-Weil rank of the Jacobian of the smooth relative curve $\Ccal \to S$ and the Mordell-Weil rank of $\Ccal|_{S_x} \to S_x$ with $S_x \inj S$ the vertical smooth projective curve constructed as the closed fiber of $S \to U$ over $x \in |U|$.

\begin{theorem}[Shioda-Tate formula for fibered surfaces]\label[theorem]{Shioda-Tate}
	Let $k$ be any field and $\Ccal \to S$ a fibered surface over $k$. Call $C$ its generic fiber over $K$, and let $\Acal \defeq \Pic^0_{\Ccal/S}$ and $B$ the $K|k$-trace of $A = \Jac(\Ccal_K)$. Then $\rk \NS(\Ccal) = 2 + \rk A(K)/B(k) + \sum_v{(h_v-1)}$ where $h_v$ is the number of $k(v)$-rational components of its fiber.
\end{theorem}
\begin{proof}
	See~\cite[Proposition~4.5 and its Corollary]{Gordon}.
\end{proof}

We first apply the Shioda-Tate formula~\cref{Shioda-Tate} to the smooth proper surface $\Ccal_x$ over $\Spec{\kappa(x)}$ fibered over the curve $S_x$:

\begin{lemma}\label[lemma]{Shioda-Tate closed point}
	Let $S_x$ be the smooth projective geometrically connected curve constructed as the closed fiber of $S \to U$ over $x \in |U|$.
	
	Then one has $\rho(\Ccal_x) = 2 + \rk \Acal(S_x)/B(k) + \sum_{v \in |S_x|}(h_v - 1)$ with $h_v$ the number of $\kappa(v)$-rational components of the fiber $\Ccal_v$ of $\Ccal_x \to S_x$ over $v \in |S_x|$.
	
	If $\Ccal \to S$ is a proper smooth relative curve, the error term $\sum_{v \in |S_x|}(h_v - 1)$ is $0$.
\end{lemma}
\begin{proof}
	The hypotheses of the Shioda-Tate formula~\cref{Shioda-Tate} are satisfied for the surface $\Ccal_x \to \{x\}$ fibered over the curve $S_x$. If $\Ccal \to S$ is a proper smooth relative curve, the error term vanishes trivially.
\end{proof}

We now apply the Shioda-Tate formula~\cref{Shioda-Tate} to the smooth proper surface $\Ccal_{k(U)} \to \Spec{k(U)}$ fibered over the curve $S_{k(U)}$:

\begin{lemma}\label[lemma]{Shioda-Tate generic point}
	Denote the function field of $S$ (equivalently, of $S_{k(U)}$) by $K$.
	
	One has $\rho(\Ccal_{k(U)}) = 2 + \rk \Acal(K)/B(k)$.
\end{lemma}
\begin{proof}
	The Shioda-Tate formula~\cref{Shioda-Tate} applied to the surface $\Ccal_{k(U)} \to \Spec{k(U)}$ fibered over the curve $S_{k(U)} \to \Spec{k(U)}$ shows that
	\[
		\rho(\Ccal_{k(U)}) = 2 + \rk \Acal(K)/B(k) + \sum_{v \in |S_{k(U)}|}(h_v - 1),
	\]
	where $h_v$ denotes the number of $\kappa(v)$-rational irreducible components of the fiber $\Ccal_v$ of $\Ccal_{k(U)} \to S_{k(U)}$.
	
	But a closed point $v \in S_{k(U)}$ of the generic fiber of $S \to U$ gives rise to a horizontal curve $S_v \to U$ by taking its closure in $S \supset S_{k(U)}$. Now if $\Ccal_v = \Ccal_{k(U)} \times_{S_{k(U)}} \{v\}$ were reducible, infinitely many of the closed fibers of the family $\Ccal|_{S_v} \to S_v$ had reducible fibers, a contradiction to $\Ccal \to S$ being a fibered surface. Hence the error term vanishes.
\end{proof}

We now compare the Mordell-Weil group of the Jacobian $\Acal = \Pic^0_{\Ccal/S}$ of the relative curve $\Ccal \to S$ to the Mordell-Weil group of the Jacobian $A$ of its generic fiber:

\begin{lemma}\label[lemma]{AcalK iso AcalS}
	Restricting a section of the abelian scheme $\Acal \to S$ to the generic point of $S$ gives an isomorphism $\Acal(S) \isoto A(K)$.
\end{lemma}
\begin{proof}
	Since $S$ is regular and $\Acal \to S$ is proper, by the valuative criterion for properness every element of $A(K)$ extends to a rational map $S \dashrightarrow \Acal$ defined outside a closed subset of codimension $\geq 2$ in $S$, i.e., the locus of indeterminacy consists of a finite set of closed points.  After a blow-up in a finite set of closed points of $S$, this becomes a morphism.  But since $S$ and the closed set is regular, the exceptional divisors are projective spaces, which admit only constant morphisms to abelian varieties.  Hence $S \dashrightarrow \Acal$ extends uniquely to a section of $\Acal \to S$, i.e., one has a homomorphism $A(K) \to \Acal(S)$ inverse to the restriction $\Acal(S) \to A(K)$.
\end{proof}

We now combine the previous results to the main result of this note, an equality between the Mordell-Weil ranks of $\Acal \to S$ and $\Acal|_{S_x} \to S_x$:

\begin{corollary}\label[corollary]{Mordell-Weil specialization rank does not grow}
	In the above situation, assume that the $K|k$-trace $B$ is trivial. After blowing up $S$ and possibly shrinking it, such that there is a morphism $S \to U$ as in~\cref{fibered surface}, there exists a $d \geq 1$ such that for infinitely many closed points $x$ of $U$ with $[\kappa(x):k] \leq d$ (the complement of a sparse subset of $|U|$), one has $\rk A(K) \geq \rk \Acal(S_x)$.
	
	If $\Ccal \to S$ is a proper smooth relative curve, one has equality $\rk A(K) = \rk \Acal(S_x)$.
\end{corollary}
(The definition of a sparse subset in this setting can be found in~\cite[Definition~1.7.1.1]{ambrosi2018specialization}.)
\begin{proof}
	One has
	\[
		\rk \Acal(S) = \rk A(K) \geq \rk \Acal(S_x),
	\]
	where the equality holds by~\cref{AcalK iso AcalS} and the inequality is a combination of~\cref{NS specialization,Shioda-Tate generic point,Shioda-Tate closed point} for all $x \in |U|$ except from the complement of a sparse subset of $|U|$.
\end{proof}

Since every abelian variety over an infinite field is a quotient of a Jacobian~\cite[Theorem~10.1]{MilneJacobianVarieties}, this easily generalizes to:

\begin{theorem}\label[theorem]{Mordell-Weil specialization}
	Let $k$ be an infinite finitely generated field. Let $K|k$ be a finitely generated regular field extension and $S$ a smooth separated (not necessarily proper) geometrically connected surface over $k$ with function field $K$. Let $A$ be an abelian variety over $K$ with $\Tr_{K|k}(A) = 0$ or $\Tr_{K|k}(A)(k)$ finite, and $\Acal$ an extension of $A$ to an abelian scheme over a dense open subscheme $U$ of $S$.
	
	Then for infinitely many curves $C$ on $U$, one has a specialization isomorphism
	\[
		A(K) \otimes \Q \isoto \Acal(C) \otimes \Q
	\]
	of rationalized Mordell-Weil groups. More precisely, for all fibrations of $S$ over a curve $U$, there is a $d \geq 1$ such that one can take infinitely many vertical curves $S_x$ for $x \in |U|$ of degree $[\kappa(x) : k] \leq d$.
\end{theorem}
\begin{proof}
	By~\cite[Theorem~10.1]{MilneJacobianVarieties} (note that $K$ is infinite), there exists a smooth projective geometrically connected curve $C$ over $K$ and a surjective homomorphism $\Pic^0_{C/K} \surj A$. Since the isogeny category of abelian varieties is semisimple (Poincar\'e's complete reducibility theorem), $\Pic^0_{C/K}$ is isogenous to a product $A \times_K B$ of abelian varieties.
	
	We use that the intersection of the set of vertical divisors $S_x$ in~\cref{Mordell-Weil specialization rank does not grow} and the divisors in~\cref{restriction to closed subscheme injective} is infinite: For $x \in |U|$ with degree $[\kappa(x):k]$ bounded, infinitely many of the $S_x$ satisfy the statement in~\cref{Mordell-Weil specialization rank does not grow}, so almost all of them are covered by~\cref{restriction to closed subscheme injective}.
	
	In the following use that the Mordell-Weil rank does not change under isogenies. By spreading out and possibly shrinking $S$, one obtains an isogeny $\Pic^0_{\Ccal/S} \to \Acal \times_S \Bcal$. By~\cref{Mordell-Weil specialization rank does not grow} for $\Ccal \to S$ and the $S_x$ there and because the rank is additive,
	\begin{align*}
		\rk \Acal(S_x) + \rk \Bcal(S_x) &\geq \rk \Acal(S) + \rk \Bcal(S) \qquad\text{by \cref{restriction to closed subscheme injective}, \eqref{eq:Silverman specialization}} \\
		&= \rk \Pic^0_{\Ccal/S}(S)  \qquad\text{by \cref{Mordell-Weil specialization rank does not grow}} \\
		&\geq \rk \Pic^0_{\Ccal/S}(S_x) \\
		&= \rk \Acal(S_x) + \rk \Bcal(S_x).
	\end{align*}
	Hence one must have equality $\rk \Acal(S) + \rk \Bcal(S) = \rk \Acal(S_x) + \rk \Bcal(S_x)$ throughout. Since $\rk \Acal(S) \leq \rk \Acal(S_x)$ and analogously for $\Bcal$, the equality $\rk \Acal(S) = \rk \Acal(S_x)$ of Mordell-Weil ranks follows.
	
	The injectivity of the rationalized specialization morphisms together with the equality of ranks implies that the rationalized specialization morphisms are isomorphisms.
\end{proof}

\noindent
\textbf{Acknowledgements.}
	I thank David Holmes and Moritz Kerz for hints and helpful discussions and Michael Stoll for proof-reading. I thank the unknown referee for helpful hints to improve the exposition.

\bibliographystyle{amsalpha}
{\footnotesize
	\bibliography{references}
}

\end{document}